\documentclass[a4paper,12pt]{article}

\usepackage{amscd}
\usepackage{amsfonts}
\usepackage{amsmath}
\usepackage{amssymb}
\usepackage{amsthm}
\usepackage[T1]{fontenc} 
\usepackage{here} 
\usepackage{mathrsfs} 
\usepackage{txfonts} 
\usepackage[all]{xy} 
\usepackage{algorithm} 
\usepackage{algpseudocode} 
\usepackage{diagbox} 
\usepackage{ulem} 

\allowdisplaybreaks

\theoremstyle{plain}
\newtheorem{thm}{Theorem}[section]
\newtheorem{lmm}[thm]{Lemma}
\newtheorem{prp}[thm]{Proposition}
\newtheorem{crl}[thm]{Corollary}
\newtheorem{conj}[thm]{Conjecture}
\theoremstyle{definition}
\newtheorem{dfn}[thm]{Definition}

\newcommand{\vs}[1][0.2]{\vspace{#1in}\noindent\ignorespaces}
\newcommand{\ba}{\begin{array*}}
\newcommand{\ea}{\end{array*}}
\newcommand{\be}{\begin{eqnarray*}}
\newcommand{\ee}{\end{eqnarray*}}
\newcommand{\bi}{\begin{itemize}}
\newcommand{\ei}{\end{itemize}}
\newcommand{\bb}{\vs\begin{itembox}}
\newcommand{\eb}{\end{itembox}}
\newcommand{\bc}{\begin{center}}
\newcommand{\ec}{\end{center}}
\newcommand{\bs}{\vs\begin{screen}}
\newcommand{\es}{\end{screen}}

\def\ens#1{{\mathchoice{\left\{ #1 \right\}}{\{ #1 \}}{\{ #1 \}}{\{ #1 \}}}}
\def\set#1#2{{\mathchoice{\left\{ #1 \ \middle| \ #2 \right\}}{\{ #1 \mid #2 \}}{\{ #1 \mid #2 \}}{\{ #1 \mid #2 \}}}}
\def\r#1{\text{\rm #1}}
\def\t#1{\text{#1}}
\def\Bigv#1{\left| #1 \right|}
\def\v#1{{\mathchoice{\Bigv{#1}}{| #1 |}{| #1 |}{| #1 |}}}

\def\ol#1{\overline{#1}{}}
\def\tl#1{\tilde{#1}{}}

\newcommand{\bC}{\mathbb{C}}

\newcommand{\bN}{\mathbb{N}}

\newcommand{\bP}{\mathbb{P}}
\newcommand{\bQ}{\mathbb{Q}}

\newcommand{\bZ}{\mathbb{Z}}

\newcommand{\cA}{\mathscr{A}}

\newcommand{\cU}{\mathscr{U}}

\newcommand{\C}{\bC}

\newcommand{\N}{\bN}
\newcommand{\Q}{\bQ}

\newcommand{\Z}{\bZ}

\newcommand{\Fp}{\mathbb{F}_p}

\algnewcommand\algorithmicbreak{{\bf break}}
\algnewcommand\Break{\algorithmicbreak{}}
\algnewcommand\algorithmiccontinue{{\bf continue}}
\algnewcommand\Continue{\algorithmiccontinue{}}

\newcommand{\NS}{\r{NS}}
\newcommand{\TS}{\r{TS}}

\title{Transcendence of Tonelli--Shanks Power Modulo Infinitely Large Primes}
\author{Tomoki Mihara}
\date{}

\begin{document}

\maketitle
\begin{abstract}
For a prime number $p$, we denote by $k_p$ the greatest odd divisor of $p-1$. Tonelli--Shanks algorithm is a classical algorithm to compute a square root of a quadratic residue $n \in \mathbb{Z}$ modulo $p$ as the product of $n^{\frac{k_p+1}{2}}$, which we call {\it Tonelli--Shanks power of $n$}, and a correction term given as a power of a quadratic non-residue modulo $p$. We prove the transcendence over $\mathbb{Q}$ of the images of the following in the ring $\mathscr{A}$ of integers modulo infinitely large primes: the greatest odd divisor $k_p$ of $p-1$, the composite $f(\frac{p - 1}{k_p})$ of any injective map $f \colon \mathbb{Z} \to \mathbb{Z}$ and the greatest $2$-power $\frac{p - 1}{k_p}$ dividing $p-1$, Tonelli--Shanks power $n^{\frac{_p+1}{2}}$ of any $n \in \mathbb{Z} \setminus \{-1,0,1\}$, and the correction term of Tonelli--Shanks algorithm for such $n$.
\end{abstract}

\tableofcontents

\section{Introduction}
\label{Introduction}

We denote by $\bP$ the set of prime numbers. Let $p \in \bP$. We denote by $k_p \in \N$ the greatest odd divisor of $p - 1$. For $k \in 2 \N + 1$, we set $\bP_k \subset \set{p \in \bP \setminus \ens{2}}{k_p = k}$. It is difficult to determine $\# \bP_k$ for $k \in \N$, and there are several number theoretic conjectures related to its infiniteness.

\begin{conj}[Fermat Prime Conjecture]
There exist infinitely many Fermat primes, i.e.\ the equation $\# \bP_1 = \infty$ holds.
\end{conj}

\begin{conj}[Thabit Prime Conjecture]
There exist infinitely many Thabit primes of second kind, i.e.\ the equation $\# \bP_3 = \infty$ holds.
\end{conj}

There is no known $k \in 2 \N + 1$ with $\# \bP_k = \infty$. Indeed, the existence of such a $k$ directly implies Proth Prime Conjecture.

\begin{conj}[Proth Prime Conjecture]
There exist infinitely many Proth primes, i.e.\ $p \in \bP$ with $k_p < \sqrt{p - 1}$.
\end{conj}

Compared to the infiniteness problem of $\# \bP_k$ for $k \in \N$, the emptiness problem has ever been successfully studied.

\begin{thm}[\cite{Sie60} Th\'eor\`eme]
There exist infinitely many Sierpi\'nski numbers, i.e.\ $k \in 2 \N + 1$ with $\# \bP_k = 0$.
\end{thm}

\begin{conj}[Sierpi\'nski conjecture]
The least Sierpi\'nski number is $78557$.
\end{conj}

\begin{thm}[\cite{EO79} Theorem 1]
There exist infinitely many odd numbers which are not Sierpi\'nski numbers, i.e.\ $k \in 2 \N + 1$ with $\# \bP_k > 0$.
\end{thm}

The aim of this paper is to study the distribution of $k_p$ from a new point of view using the ring
\be
\cA \coloneqq \prod_{p \in \bP} \Fp \bigg/ \bigoplus_{p \in \bP} \Fp
\ee
of integers modulo infinitely large primes. In number theory, $\cA$ plays a role of the base ring in the study of finite multiple zeta values, which are finite analogues of multiple zeta values introduced by D.\ Zagier. Starting from \cite{KZ}, there have been various studies on specific elements of $\cA$. We give new transcendence results on $k_p$:

\begin{thm}
\label{intro transcendence of v_2}
For any injective map $f \colon \N \to \Z$, the image of $(f(\frac{p - 1}{k_p}))_{p \in \bP}$ in $\cA$ is transcendental over $\Q$.
\end{thm}

\begin{thm}
\label{intro transcendence of u_2}
The image of $(k_p)_{p \in \bP}$ in $\cA$ is transcendental over $\Q$.
\end{thm}

The amount $k_p$ is used in Tonelli--Shanks algorithm, which is an effective method to compute a square root modulo $p$. When an input $n \in \Z$ is a quadratic residue modulo $p$, Tonelli--Shanks algorithm returns a square root of $n$ modulo $p$ given as the product of $n^{\frac{k_p + 1}{2}}$ and a correction term given as a power of a quadratic non-residue modulo $p$. We call $n^{\frac{k_p + 1}{2}}$ {\it Tonelli--Shanks power of $n$}.

\vs
When the correction term is $1$, then the algebraic relation $(n^{\frac{k_p + 1}{2}})^2 \equiv n \pmod{p}$ holds. In general (even if $n$ is a quadratic non-residue modulo $p$), the algebraic relation $(n^{\frac{k_p + 1}{2}})^{2^{v+1}} \equiv n^{2^v} \pmod{p}$ holds for a sufficiently large $v \in \N$, e.g.\ $v = \log_2 \frac{p - 1}{k_p}$. Then natural questions arise:
\bi
\item[(1)] When $p$ varies, how often can $n^{\frac{k_p + 1}{2}}$ be a square root of $n$ modulo $p$?
\item[(2)] When $p$ varies, is there a law which controls the least $v \in \N$ such that $(n^{\frac{k_p + 1}{2}})^{2^{v+1}} \equiv n^{2^v} \pmod{p}$?
\ei
Such questions can be interpreted into algebraic nature of the image of $(n^{\frac{k_p + 1}{2}})_{p \in \bP}$ in $\cA$. For example, if $(n^{\frac{k_p + 1}{2}})^2 \equiv n \pmod{p}$ holds for all but finitely many $p \in \bP$, then the image of $(n^{\frac{k_p + 1}{2}})_{p \in \bP}$ in $\cA$ is a root of $x^2 - n \in \Q[x]$, and hence is algebraic over $\Q$. However, if $n$ is not a square number, then there exist infinitely may $p \in \bP$ such that $n + p \Z$ is a quadratic non-residue, and hence it is not the case. If there is an upperbound $V$ of the least $v$ independent of $p \in \bP$, then the image of $(n^{\frac{k_p + 1}{2}})_{p \in \bP}$ in $\cA$ is a root of $x^{2^{V+1}} - n^{2^V} \in \Q[x]$, and hence is algebraic over $\Q$. However, it is not a case unless $\v{n} \leq 1$:

\begin{thm}
\label{intro transcendence of Tonelli--Shanks power}
For any $n \in \Z$, the following are equivalent:
\bi
\item[(1)] The image of $(n^{\frac{k_p + 1}{2}})_{p \in \bP}$ in $\cA$ is transcendental over $\Q$.
\item[(2)] The inequality $\v{n} > 1$ holds.
\ei
\end{thm}

We briefly explain contents of this paper. In \S \ref{Convention}, we introduce convention in this paper. In \S \ref{Non-standard Limit}, we extend the notion of a non-standard limit originally formulated in \cite{Mih26-2}, and apply it to the study of transcendence. In \S \ref{Naive Transcendence of Greatest Odd Divisor of p-1}, we show Theorem \ref{intro transcendence of v_2} and Theorem \ref{intro transcendence of u_2}. In \S \ref{Naive Transcendence of Tonelli--Shanks Power}, we show Theorem \ref{intro transcendence of Tonelli--Shanks power}.

\section{Convention}
\label{Convention}

We denote by $\N$ the set of non-negative integers, and by $\bP$ the set of prime numbers. For $p \in \bP$, we denote by $\Fp$ the finite field $\Z/p \Z$.

\vs
For a set $I$, we denote by $\# I$ its cardinality. For sets $X$ and $Y$, we denote by $X^Y$ the set of maps $Y \to X$. When we handle a sequence $s$ indexed by a set $I$, we frequently use the map notation $s(i)$ instead of the subscript notation $s_i$ to point the entry at $i \in I$, in order to avoid massive use of subscripts. For a set $X$, $x \in X$, and a binary relation $R$ on $X$, we set $X_{R x} \coloneqq \set{x' \in X}{x' R x}$. For example, every $d \in \N$ is identical to $\N_{< d} = \set{i \in \N}{i < d}$, and hence for a set $X$, $X^d$ formally means $X^{\N_{< d}}$, which is naturally identified with the set of $d$-tuples in $X$.

\vs
For $p \in \bP$, we denote by $\Z_{(p)} \subset \Q$ the localisation of $\Z$ at the prime ideal $p \Z$. We define a ternary relation $q_0 \equiv q_1 \pmod{p}$ on $(p,q_0,q_1) \in \bP \times \Q^2$ by
\be
q_0 - q_1 \in p \Z_{(p)} \land q_1 \in \Z_{(p)},
\ee
and denote by $q_0 \not\equiv q_1 \pmod{p}$ its negation. For $(p,q) \in \bP \times \Q$, we set
\be
q \bmod p \coloneqq
\left\{
\begin{array}{ll}
\set{n \in \Z}{q \equiv n \pmod{p}} & (q \in \Z_{(p)}) \\
p \Z & (q \notin \Z_{(p)})
\end{array}
\right.
\in \Fp.
\ee
We note that the value for the latter case does not affect the theory, as we use it only to ignore conventionally exceptional cases. We set
\be
\tl{\cA} & \coloneqq & \set{a \in \Q^{\bP}}{\# \set{p \in \bP}{a(p) \notin \Z_{(p)}} < \infty}.
\ee
For $a \in \tl{\cA}$, we call
\be
(a(p) \bmod p)_{p \in \bP} + \bigoplus_{p \in \bP} \Fp \in \cA
\ee
{\it the image of $a$ in $\cA$}. The reader should be careful that the convention $\tl{\cA}$ is different from that of \cite{Mih26-2}, which considers a sequence indexed by $\N$ rather than $\bP$.
\vs
We say that $a \in \cA$ is {\it naively algebraic} if there exists an $f \in \Q[x] \setminus \ens{0}$ such that $f(a) = 0$, and is {\it naively transcendental} if $a$ is not naively algebraic. Here, we formally define $0^0 \coloneqq 1$.

\section{Non-standard Limit}
\label{Non-standard Limit}

We introduced the predicate that $q \in \Q$ is a non-standard limit of a given $a \in \tl{\cA}$ in \cite{Mih26-2} Definition 2.1. We extend the definition to $q \in \C$ through the equivalence of \cite{Mih26-2} Proposition 2.2 (1) and (2):

\begin{dfn}
Let $a \in \tl{\cA}$ and $q \in \C$. We say that $q$ is a {\t non-standard limit of $a$} if there exists a ring homomorphism $\phi \colon \cA \to \C$ sending the image of $a$ in $\cA$ to $q$. We denote by $\NS(a) \subset \C$ the subset of non-standard limits of $a$ algebraic over $\Q$.
\end{dfn}

We have an analogue of the equivalence of \cite{Mih26-2} Proposition 2.2 (2) and (3).

\begin{prp}
\label{minimal polynomial vanishing}
Let $a \in \tl{\cA}$ and $q \in \C$. If $q$ is an algebraic number, then the following are equivalent:
\bi
\item[(1)] The relation $q \in \NS(a)$ holds.
\item[(2)] There exist infinitely many $p \in \bP$ such that $f(a(p)) \equiv 0 \pmod{p}$, where $f$ denotes the minimal polynomial of $q$ over $\Q$. 
\ei
\end{prp}

\begin{proof}
We denote by $\ol{a} \in \cA$ the image of $a$ in $\cA$. Set
\be
S \coloneqq \set{p \in \bP}{f(a(p)) \equiv 0 \pmod{p}}.
\ee
Assume (1). Take a ring homomorphism $\phi \colon \cA \to \C$ sending $\ol{a}$ to $q$. By \cite{Mih26-1} Proposition 2.4 (2) and \cite{Mih26-1} Proposition 3.1, there exists a non-principal ultrafilter $\cU$ of $\bP$ such that $\ker(\phi)$ coincides with $\set{a' \in \cA}{\lim_{\cU} a' = 0}$. We have
\be
\phi(f(\ol{a})) = f(\phi(\ol{a})) = f(q) = 0,
\ee
and hence $f(\ol{a}) \in \ker(\phi)$. By the choice of $\cU$, we obtain $\lim_{\cU} f(\ol{a}) = 0$, i.e.\ $S \in \cU$. Since $\cU$ is non-principal, this implies $\# S = \infty$.

\vs
Assume (2). By $\# S = \infty$, there exists a non-principal ultrafilter $\cU$ of $\bP$ such that $S \in \cU$, i.e.\ $\lim_{\cU} f(\ol{a}) = 0$. By \cite{Mih26-1} Proposition 3.1, there exists a ring homomorphism $\phi \colon \cA \to \C$ such that $\ker(\phi)$ coincides with $\set{a' \in \cA}{\lim_{\cU} a' = 0}$. In particular, we have $f(\ol{a}) \in \ker(\phi)$, and hence
\be
f(\phi(\ol{a})) = \phi(f(\ol{a})) = 0.
\ee
This implies that $\phi(\ol{a})$ is conjugate to $q$ over $\Q$. Since every field automorphism of the algebraic closure of $\Q$ in $\C$ extends to a field automorphism of the algebraically closed field $\C$, there exists a field automorphism $\sigma \colon \C \to \C$ such that $\sigma(\phi(\ol{a})) = q$. Since the ring homomorphism $\sigma \circ \phi \colon \cA \to \C$ sends $\ol{a}$ to $q$, we have $q \in \NS(a)$.
\end{proof}

As a corollary, we obtain an extension of \cite{MS26} Lemma 2.1 (cf. \cite{AF24} Proposition 3.7):

\begin{crl}
\label{infinitely many limits}
Let $a \in \tl{\cA}$. If $\# \NS(a) = \infty$, then the image of $a$ in $\cA$ is naively transcendental.
\end{crl}

\begin{proof}
We denote by $\ol{a} \in \cA$ the image of $a$ in $\cA$. We show $f(\ol{a}) \neq 0$ for any $f \in \Q[x] \setminus \ens{0}$. By $f \neq 0$ and $\# \NS(a) = \infty$, there exists some $q \in \NS(a)$ such that $f(q) \neq 0$. Set
\be
S \coloneqq \set{p \in \bP}{g(a(p)) \equiv 0 \pmod{p}},
\ee
where $g \in \Q[x]$ denotes the minimal polynomial of $q$ over $\Q$. By $q \in \NS(a)$ and Proposition \ref{minimal polynomial vanishing}, we have $\# S = \infty$. Therefore, it suffices to show $f(a(p)) \not\equiv 0 \pmod{p}$ for all but finitely many $p \in S$.

\vs
By $f(q) \neq 0$ and the minimality of $g$, we have $\gcd(f,g) = 1$ in $\Q[x]$, i.e.\ there exists some $(s,t) \in \Q[x]^2$ such that $s f + t g = 1$. For any $p \in S$, we have
\be
s(a(p)) f(a(p)) = s(a(p)) f(a(p)) + t(a(p)) g(a(p)) = 1 \not\equiv 0 \pmod{p},
\ee
and hence $f(a(p)) \not\equiv 0 \pmod{p}$ as long as $s \in \Z_{(p)}[x]$.
\end{proof}

\begin{crl}
\label{infinitely many functions}
Let $a \in \tl{\cA}$. If there are infinitely many irreducible $f \in \Q[x]$ such that $\# \set{p \in \bP}{f(a(p)) \equiv 0 \pmod{p}} = \infty$, then the image of $a$ in $\cA$ is naively transcendental.
\end{crl}

\begin{proof}
By Corollary \ref{infinitely many limits}, it suffices to show $\# \NS(a) = \infty$. For any irreducible $f \in \Q[x]$ such that $\# \set{p \in \bP}{f(a(p)) \equiv 0 \pmod{p}} = \infty$, we have $\set{q \in \C}{f(q) = 0} \subset \NS(a)$ by Proposition \ref{minimal polynomial vanishing}. By the infinite existence of such a monic $f$, we have $\# \NS(a) = \infty$.
\end{proof}

\begin{crl}
\label{infinitely many values}
Let $X \subset \Z$ and $a \in X^{\bP}$. If there exist infinitely many $q \in X$ such that $\# \set{p \in \bP}{a(p) = q} = \infty$, then for any injective map $f \colon X \to \Z$, the image of $f \circ a \in \Z^{\bP} \subset \tl{\cA}$ in $\cA$ is naively transcendental.
\end{crl}

\begin{proof}
It suffices to show $\# \NS(f \circ a) = \infty$ by Corollary \ref{infinitely many limits}. Set
\be
X' \coloneqq \set{q \in X}{\# \set{p \in \bP}{a(p) = q} = \infty}.
\ee
We have $\# X' = \infty$ by the assumption. For any $q \in X'$, we have
\be
\# \set{p \in \bP}{f(a(p)) \equiv f(q) \pmod{p}} \geq \# \set{p \in \bP}{f(a(p)) = f(q)} \geq \# \set{p \in \bP}{a(p) = q} = \infty
\ee
by $f(q) \in \Z$. This implies $f(q) \in \NS(f \circ a)$. We obtain
\be
\# \NS(f \circ a) \geq \# f(X') = \# X' = \infty
\ee
by the injectivity of $f$.
\end{proof}

\section{Naive Transcendence of Greatest Odd Divisor of $p-1$}
\label{Naive Transcendence of Greatest Odd Divisor of p-1}

In this section, we study the naive transcendence of the greatest odd divisor of $p-1$ for $p \in \bP$.

\begin{dfn}
For $n \in \N \setminus \ens{0}$, we denote by $v_2(n) \in \N$ the additive $2$-adic valuation of $n$, and by $u_2(n)$ the greatest odd divisor $2^{-v_2(n)}n \in 2 \N + 1$ of $n$.
\end{dfn}

We regard $u_2(x-1)$ and $v_2(x-1)$ as functions $\bP \to \N$. We restate and prove Theorem \ref{intro transcendence of v_2}.

\begin{thm}
\label{transcendence of v_2}
For any injective map $f \colon \N \to \Z$, the image of $f(v_2(x-1))$ in $\cA$ is naively transcendental.
\end{thm}

\begin{proof}
For any $v \in \N$, we have
\be
\# \set{p \in \bP}{v_2(p-1) = v} = \# \set{p \in \bP}{p \equiv 1 + 2^v \pmod{2^{v+1}}} = \infty
\ee
by Dirichlet's theorem on arithmetic progression. Therefore, the assertion follows from Corollary \ref{infinitely many values} applied to $(X,a) = (\N,v_2(x-1))$.
\end{proof}

We restate and prove Theorem \ref{intro transcendence of u_2}.

\begin{thm}
\label{transcendence of u_2}
The image of $u_2(x-1)$ in $\cA$ is naively transcendental.
\end{thm}

\begin{proof}
By Theorem \ref{transcendence of v_2}, the image of $2^{u_2(x-1)}$ in $\cA$ is naively transcendental. For any $p \in \bP$, we have $p = u_2(p-1) 2^{v_2(p-1)} + 1$ and hence
\be
u_2(p-1) \equiv - 2^{- v_2(p-1)} \pmod{p}.
\ee
This implies the assertion.
\end{proof}

\section{Naive Transcendence of Tonelli--Shanks Power}
\label{Naive Transcendence of Tonelli--Shanks Power}

In this section, we study the naive transcendence of Tonelli--Shanks power.

\begin{dfn}
For $n \in \N$, we define a map $\TS_n \colon \bP \to \Z$ by
\be
\TS_n(p) = n^{\frac{u_2(p-1) + 1}{2}},
\ee
and call $\TS_n$ {\it Tonelli--Shanks power of $n$}.
\end{dfn}

We restate and later prove Theorem \ref{intro transcendence of Tonelli--Shanks power}.

\begin{thm}
\label{transcendence of Tonelli--Shanks power}
For any $n \in \Z$, the following are equivalent:
\bi
\item[(1)] The image of $\TS_n$ in $\cA$ is naively transcendental.
\item[(2)] The inequality $\v{n} > 1$ holds.
\ei
\end{thm}

In order to prove Theorem \ref{transcendence of Tonelli--Shanks power}, we prepare convention and lemmata. For a subset $B \subset \bP$ with Dirichlet density, we denote by $\Delta(B) \in [0,1]$ Dirichlet density of $B$. Let $n \in \Z$ and $v \in \N$. We set
\be
B_n & \coloneqq & \set{p \in \bP}{n^{u_2(p-1)} \not\equiv 1 \pmod{p}} \\
B_{n,v} & \coloneqq & \set{p \in B_n}{v_2(p - 1) = v}.
\ee
If $(n,v) = (0,0)$, then we have $B_{n,v} = \ens{2}$. If $n = 0$ and $v > 0$, then we have $B_{n,v} = \set{p \in \bP}{p \equiv 1 + 2^v \pmod{2^{v+1}}}$, and hence $\# B_{n,v} = \infty$ by Dirichlet's theorem on arithmetic progression. If $n \neq 0$, then $B_{n,v}$ is identical to $B_v$ for $a = n$ with respect to the convention in \cite{Has66} by the definition of $u_2$, and hence has Dirichlet density.

\begin{lmm}
\label{Dirichlet density}
For any $(n,v) \in (\Z \setminus \ens{-1,0,1}) \times \N_{> 2}$, $\Delta(B_{n^{2^{v+1}}}) < \Delta(B_{n^{2^v}})$ holds.
\end{lmm}

\begin{proof}
The explicit formula of $\Delta(B_v)$ and the table of $\kappa_v$ for $v \in \N$ in \cite{Has66} imply $\Delta(B_{n^{2^v}}) < 1$ and $\Delta(B_{n^{2^{v+1}}}) = 2^{-1}(1 + \Delta(B_{n^{2^v}}))$ for $v \in \N$. These imply the assertion.
\end{proof}

\begin{lmm}
\label{root of unity}
For any $(n,v) \in (\Z \setminus \ens{-1,0,1}) \times \N_{> 3}$, there exist infinitely many $p \in \bP$ such that $\Phi_{2^v}(n^{u_2(p-1)}) \equiv 0 \pmod{p}$, where $\Phi_{2^v}$ denotes the $2^v$-th cyclotomic polynomial.
\end{lmm}

\begin{proof}
The assertion immediately follows from Lemma \ref{Dirichlet density} applied to $(n,v-1)$, because for any $p \in \bP$, the system
\be
\left\{
\begin{array}{rcl}
n^{u_2(p-1) 2^{v-1}} & \not\equiv & 1 \pmod{p} \\
n^{u_2(p-1) 2^v} & \equiv & 1 \pmod{p}
\end{array}
\right.
\ee
of congruence relations is equivalent to both of $p \in B_{n^{2^{v-1}}} \setminus B_{n^{2^v}}$ and $\Phi_{2^v}(n^{u_2(p-1)}) \equiv 0 \pmod{p}$.
\end{proof}

\begin{lmm}
\label{u_2-th power}
For any $n \in \Z$, the following are equivalent:
\bi
\item[(1)] The image of $n^{u_2(x-1)}$ in $\cA$ is naively transcendental.
\item[(2)] The inequality $\v{n} > 1$ holds.
\ei
\end{lmm}

\begin{proof}
We have $(n^{u_2(x-1)})^2 = n^2$ if $\v{n} \leq 1$. Therefore, we may assume $\v{n} > 1$. By Corollary \ref{infinitely many functions} and Lemma \ref{root of unity}, the image of $n^{u_2(x-1)}$ in $\cA$ is naively transcendental.
\end{proof}

\begin{proof}[Proof of Theorem \ref{transcendence of Tonelli--Shanks power}]
The assertion immediately follows from Lemma \ref{u_2-th power} and $\TS_n(x)^2 = n \times n^{u_2(x-1)}$.
\end{proof}

For $n \in \Z$, we say that $r \in \tl{\cA}$ is a {\it residual square root of $n$} if there exists some $q \in \Q^{\times}$ such that for any $p \in \bP$, the relation
\be
\left( \left( \frac{r}{p} \right) = -1 \Rightarrow r(p) \equiv q \pmod{p} \right) \land \left( \left( \frac{r}{p} \right) \neq -1 \Rightarrow r(p)^2 \equiv n \pmod{p} \right)
\ee
holds, where $(\frac{r}{p})$ denotes Legendre symbol. The correction term of Tonelli--Shanks algorithm forms a naively transcendental element of $\cA$ in the following sense:

\begin{crl}
For any $f \in \tl{\cA}$, if $\TS_n f$ is a residual square root $r$ of $n$ for some $n \in \Z$ with $\v{n} > 1$, then the image of $f$ in $\cA$ is naively transcendental.
\end{crl}

\begin{proof}
We denote by $a \in \cA$ the image of $\TS_n$, and by $\ol{f} \in \cA$ the image of $f$. Since $\TS_n f$ is a residual square root $r$ of $n$, we have $((a \ol{f})^2 - n)(a \ol{f} - q) = 0$ for some $q \in \Q^{\times}$. Therefore, $a \ol{f}$ is naively algebraic. By $nq \in \Q^{\times}$, $a \ol{f}$ is invertible, and we have $((a \ol{f})^{-2} - n^{-1})((a \ol{f})^{-1} - q^{-1}) = 0$. Therefore, $(a \ol{f})^{-1}$ is naively algebraic. If $\ol{f}$ is naively algebraic, then so is $\ol{f}(a \ol{f})^{-1} = a^{-1}$. This contradicts the naive transcendence of $a$, which follows from $\v{n} > 1$ and Theorem \ref{transcendence of Tonelli--Shanks power}. Therefore, $\ol{f}$ is naively transcendental.
\end{proof}

\vspace{0.3in}
\addcontentsline{toc}{section}{Acknowledgements}
\noindent {\Large \bf Acknowledgements}
\vspace{0.2in}

\noindent
I thank all people who helped me to learn mathematics and programming. I also thank my family.

\end{document}